\documentclass[english]{extarticle}
\usepackage[T1]{fontenc}
\usepackage[latin9]{inputenc}
\usepackage{textcomp}
\usepackage{mathrsfs}
\usepackage{mathtools}
\usepackage{amsmath}
\usepackage{amsthm}
\usepackage{amssymb}
\usepackage{stackrel}
\usepackage{setspace}

\makeatletter
\numberwithin{equation}{section}
\numberwithin{figure}{section}
\theoremstyle{plain}
\newtheorem{thm}{\protect\theoremname}[section]
\theoremstyle{plain}
\newtheorem{prop}[thm]{\protect\propositionname}
\theoremstyle{remark}
\newtheorem{rem}[thm]{\protect\remarkname}
\theoremstyle{definition}
\newtheorem{example}[thm]{\protect\examplename}
\ifx\proof\undefined
\newenvironment{proof}[1][\protect\proofname]{\par
	\normalfont\topsep6\p@\@plus6\p@\relax
	\trivlist
	\itemindent\parindent
	\item[\hskip\labelsep\scshape #1]\ignorespaces
}{%
	\endtrivlist\@endpefalse
}
\providecommand{\proofname}{Proof}
\fi
\theoremstyle{plain}
\newtheorem{cor}[thm]{\protect\corollaryname}
\theoremstyle{plain}
\newtheorem{lem}[thm]{\protect\lemmaname}

\@ifundefined{date}{}{\date{}}

\makeatother

\usepackage{babel}
\providecommand{\corollaryname}{Corollary}
\providecommand{\examplename}{Example}
\providecommand{\lemmaname}{Lemma}
\providecommand{\propositionname}{Proposition}
\providecommand{\remarkname}{Remark}
\providecommand{\theoremname}{Theorem}

\begin{document}
\begin{flushleft}
Immanuel Ben Porat, Einstein Institute of Mathematics, The Hebrew
University of Jerusalem, Jerusalem, Israel. email: \texttt{immanuel.benporat@mail.huji.ac.il}
\par\end{flushleft}

\medskip{}

\begin{onehalfspace}
\begin{center}
\textbf{\large{}CONVEXITY IN MULTIVALUED HARMONIC FUNCTIONS}{\large\par}
\par\end{center}
\end{onehalfspace}
\begin{abstract}
We investigate variants of a three circles type Theorem in the context
of $\mathcal{Q}-$valued functions. We prove some convexity inequalities
related to the $L^{2}$-growth function in the $\mathcal{Q}-$valued
settings. Optimality of these inequalities and comparsion to the case
of real valued harmonic functions is also discussed. 
\end{abstract}
\begin{doublespace}

\section{Introduction}
\end{doublespace}

\subsection{Background}

The study of multi valued harmonic functions was originated in the
pioneering work of Almgren {[}2{]} on Plateau's problem, which asks
for a surface of minimal area among all surfaces stretched across
a given closed contour. Almgren's theory was further extended and
simplified in {[}5{]}. The profound geometric applications of Almgren's
theory to minimal surfaces are not addressed here. Instead, we shall
connect the theory of $\mathcal{Q}-$valued functions to a classical
results from complex analysis, which has some modern reflections.
Let us begin by providing some background to the material that motivated
this work. Let $0\neq u$ be harmonic on the unit ball $B_{1}(0)$.
Then one can associate to $u$ real valued functions $H_{u},D_{u},\overline{H}_{u},I_{u}:\mathrm{(0,1)\rightarrow\mathbb{R}}$
by letting 
\[
H_{u}(r)\coloneqq\underset{\partial B_{r}(0)}{\int}u^{2}(x)d\sigma,
\]

\[
D_{u}(r)\coloneqq\underset{B_{r}(0)}{\int}|\nabla u|^{2}dx,
\]

\[
\overline{H}_{u}(r)\coloneqq\frac{1}{|\partial B_{r}(0)|}H_{u}(r),
\]

and 

\[
I_{u}(r)\coloneqq\frac{rD_{u}(r)}{H_{\mathit{u}}(r)}.
\]

$\overline{H}_{u}(r)$ is called the $L^{2}-$growth function of $u$
and $I_{u}(r)$ is called the frequency function of $u$. The motivation
behind the definition of the function $\overline{H}_{u}$ comes from
a classical result in complex analysis known as the three circles
Theorem:

Given a holomorphic function $f$ on the unit ball $B_{1}$, let $M(r)\coloneqq\underset{B_{r}(0)}{\max}|f|$.
The three circles Theorem, proved by Hadamard, states that the function
$\widetilde{M}:(-\infty,0)\rightarrow\mathbb{R}$ defined by $\widetilde{M}(t)\coloneqq M(e^{t})$
is $\log$ convex, that is $\log\widetilde{M}(t)$ is convex. It is
therefore natural to seek for a three circles type Theorem for real
harmonic functions $u:B_{1}(0)\subset\mathbb{R}^{n}\rightarrow\mathbb{R}$.
It was first observed by Agmon {[}1{]} that such a theorem holds
if the function $M$ is replaced by an appropriate $L^{2}$-version
on the sphere. Namely, Agmon proves that the function $t\mapsto\overline{H}_{u}(e^{t})$
is $\log$ convex. 

In 2015, Lippner and Mangoubi observed the following stronger result:
\begin{thm}
\label{Theorem 1.1 } \textup{( {[}6{]}, Theorem 1.4)}. Let $u:B_{1}(0)\rightarrow\mathbb{R}$
be harmonic. Then $\overline{H}_{u}$ is absolutely monotonic, that
is $\overline{H}_{u}^{(k)}\geq0$ for all $k\in\mathbb{N}$. In particular,
$H_{u}^{(k)}\geq0$ for all $k\in\mathbb{N}$.
\end{thm}
Since Theorem (\ref{Theorem 1.1 }) in {[}6{]} is carried out in
a discrete setting, for the sake of completeness a proof of the continuous
version (as stated above) is presented in the appendix. The second
statement in Theorem (\ref{Theorem 1.1 }) is an immediate consequence
of the first one. It is an exercise to verify that absolute monotinicity
of $\overline{H}$ implies $\log$ convexity of $t\mapsto\overline{H}(e^{t})$
(see {[}7{]}, II, Problem 123). Roughly speaking, we are interested
in the question whether a Lippner-Mangoubi type theorem can be obtained
in the more general setting of multi valued harmonic functions. Let
us emphasize this could have fascinating applications in the regularity
theory of these objects, since absolutely monotonic functions are
real analytic (due to a celebrated theorem of Bernstein. See {[}3{]}).
In some sense, the nonlinear nature of the problem is the main obstacle
in obtaining elliptic regularity type results for multi valued harmonic
functions. We hope that approaching the problem via Bernstein's theorem
may be useful in overcoming some of the difficulties that are created
by the lack of linearity.

\subsection{Main Results }

Given some $P\in\mathbb{R}^{n}$ we denote by $[[P]]$ the Dirac mass
in $P\in\mathbb{R}^{n}$ and define 
\[
\mathcal{A}_{\mathcal{Q}}(\mathbb{R}^{n}):=\{\stackrel[i=1]{\mathfrak{\mathcal{Q}}}{\sum}[[P_{i}]]|P_{i}\in\mathbb{R}^{n},1\le i\le\mathcal{Q}\}.
\]
The set $\mathcal{A}_{\mathcal{Q}}(\text{\ensuremath{\mathbb{R}}}^{n})$
is endowed with a metric $\mathcal{G}$, not specified for the moment,
such that the space $(\mathcal{A}_{\mathcal{Q}}(\text{\ensuremath{\mathbb{R}}}^{n}),\mathcal{G})$
is a complete metric space. We then consider functions $f:\Omega\subset\mathbb{R}^{m}\rightarrow\mathcal{A}_{\mathcal{Q}}(\text{\ensuremath{\mathbb{R}}}^{n})$,
where $\Omega$ is some domain in $\mathbb{R}^{m}$. We call such
functions $\mathcal{Q}$-valued functions. One key fact is the existence
of a notion of a\textsl{ harmonic $\mathcal{Q}$-valued function}.

We adapt the terminology of {[}5{]} and call such functions Dir-minimizing.
As their name suggests, Dir-minimizing functions are defined as functions
minimizing a certain class of integrals, by analogy with the classical
Dirichlet principle. For each $f:B_{1}(0)\subset\mathbb{R}^{m}\rightarrow\mathcal{A}_{\mathcal{Q}}(\mathbb{R}^{n})$
Dir-minimizing we associate a real valued function $\overline{H}_{f}:(0,1)\rightarrow\mathbb{R}$
by letting 
\[
\overline{H}_{f}(r)=\frac{1}{|\partial B_{r}(0)|}\underset{\partial B_{r}(0)}{\int}|f|{}^{2}d\sigma.
\]
The function $\overline{H}_{f}$ is a generalization of the function
introduced in the beginning. Our first aim would be to generalize
Agmon's Theorem to the multi valued case. We will prove 
\begin{prop}
\textsl{\label{Proposition 1.2} }Let $f:B_{1}(0)\subset\mathbb{R^{\mathit{m}}}\rightarrow\mathcal{A}_{\mathcal{Q}}(\text{\ensuremath{\mathbb{R}}}^{n})$
be Dir minimizing such that $H(r)>0$. Define $a:(-\infty,0)\rightarrow\mathbb{R}$
by $a(t)=\log\overline{H}(e^{t})$. Then 

(i) $a'(t)\geq0$ for all $t\in(-\infty,0)$. 

(ii) $a''(t)\geq0$ for a.e. $t\in(-\infty,0)$.

Furthermore, $a$ is convex. 
\end{prop}
Since (ii) holds merely up to a null set, the convexity of $a$ does
not follow directly. This requires an additional consideration. The
following theorem, is the main result of this work
\begin{thm}
\textsl{\label{Theorem 1.2} }Let $f:B_{1}(0)\subset\mathbb{R}^{2}\rightarrow\mathcal{A_{Q}}(\mathbb{R}^{n})$
be a Dir minimizing function. Suppose $f|_{\partial B_{r}}\in W^{1,2}(\partial B_{r}(0),\mathcal{A_{Q}}(\text{\ensuremath{\mathbb{R}}}^{n}))$
for a.e. $0<r<1$. For each $N>0$ define $\overline{h}_{N,f}:(0,1)\rightarrow\mathbb{R}$
by $\overline{h}_{N,f}(r)\coloneqq\overline{H}_{f}(r^{N}).$ Then 

(i) $\overline{h}'_{\frac{\mathcal{Q}}{2},f}(r)\geq0$ for all $r\in(0,1)$.

(ii) $\overline{h}''_{\frac{\mathcal{Q}}{2},f}(r)\geq0$ for a.e.
$r\in(0,1)$. 

Furthermore, $\overline{h}_{\frac{\mathcal{Q}}{2}}$is convex. 
\end{thm}
The proof of Theorem (\ref{Theorem 1.2}) will be followed by a higher
dimensional version, that is, when the domain is the $m-$dimensional
unit ba\d{l}l for arbitrary $m>2$. In the higher dimensional version,
the constant $\frac{\mathcal{Q}}{2}$ will be replaced by some constant
depending on $m$ which does not have a simple closed formula. It
should be remarked that unlike in the scenario of Theorem (\ref{Theorem 1.1 }),
the fact that $\overline{H}_{f}$ (and hence $a$ and $\overline{h}_{N,f}$
) is a.e. twice differentiable (and moreover $C^{1}$) is nontrivial. 

A naive version of Theorem (\ref{Theorem 1.1 }) for $\mathcal{Q}$-valued
functions is not valid, as witnessed by the example $f(z)=\underset{w^{3}=z}{\sum}[[w]]$,
for which the associated $\overline{H}$ function has a negative second
derivative for all $0<r<1$. In addition, the following proposition
demonstrates that we do not have an obvious third derivative version
of Theorem (\ref{Theorem 1.2}):
\begin{prop}
\textup{(\cite{4}) }\textsl{\label{Proposition 1.4} }Define $f:B_{1}(0)\subset\mathbb{R}^{2}\mathbb{\rightarrow\mathcal{A}_{\mathrm{2}}\mathrm{(\mathbb{R^{\mathrm{2}}\mathrm{)}}}}$
by $f(z)=\underset{w^{2}=2z-1}{\sum}[[w]]$. Then $f$ is Dir minimizing,$f|_{\partial B_{r}}\in W^{1,2}(\partial B_{r},\mathcal{A_{Q}}(\text{\ensuremath{\mathbb{R}}}^{n}))$
for all $r\neq\frac{1}{2}$ and $\overline{h}'''_{1,f}(r)=\mathcal{\mathit{\overline{H}_{f}'''}}(r)<0$
for all $\frac{1}{2}<r<1$.
\end{prop}
Both Proposition (\ref{Proposition 1.4}) and the above mentioned
example will be proved and explained in section \ref{Section 5}.
In Proposition (\ref{Proposition 1.4}), our main contribution is
performing the computation which shows that $\overline{H}_{f}'''<0$
for all $\frac{1}{2}<r<1$ and showing that the boundary condition
$f|_{\text{\ensuremath{\partial}}B_{r}}\text{\ensuremath{\in}}W^{1,2}(\text{\ensuremath{\partial}}B_{r},\mathcal{A_{Q}}(\text{\ensuremath{\mathbb{R}}}^{n}))$
is indeed satisfied for all $r\neq\frac{1}{2}$. Proving that $f$
is Dir minimizing is a difficult task, and relies on some rather heavy
machinery from geometric measure theory. It should be emphasized that
the domain of $f$ in both counterexamples is the planar unit disk.
Thus, we did not rule out the possibility that in higher dimensions
the $L^{2}$- growth function of $f$ is more well behaved.

\textit{Organization of the paper. }In section \ref{Section 3}, we
fix some notation and briefly review the frequency function and its
relatives. A detailed exposition may be found in {[}5{]}. Section
\ref{Section 4} is devoted mainly to the proof of Proposition (\ref{Proposition 1.2}),
Theorem (\ref{Theorem 1.2}) and other related convexity inequalities.
In Section \ref{Section 5} we preform the calculations required to
establish the counterexample given in Proposition (\ref{Proposition 1.4}).
In the same context we will prove that the boundary condition $f|_{\text{\ensuremath{\partial}}B_{r}}\text{\ensuremath{\in}}W^{1,2}(\text{\ensuremath{\partial}}B_{r}(0),\mathcal{A_{Q}}(\text{\ensuremath{\mathbb{R}}}^{n}))$
in Theorem (\ref{Theorem 1.2}) is in fact verified for a certain
class of Dir-minimizing functions on the unit disk. 

\begin{doublespace}

\section{\label{Section 3} Preliminaries}
\end{doublespace}

\subsection{Notations}

$d\sigma=$The surface measure.

$\Delta u=$The Laplacian of $u$.

$\cdot=$The standard scalar product on $\mathbb{R}^{m}$.

$\nu=$The unit normal to the sphere.

$C_{m}=$Surface area of the $m-$dimensional unit sphere.

$B_{R}(x)=$The ball of radius $R$ centered at $x$. 

We assume that the reader is familiar with the basic theory of $\mathcal{Q}$-valued
functions. Following {[}5{]}, we recall some basic notions and terminology.
Given some $P\in\mathbb{R}^{n}$ we denote by $[[P]]$ the Dirac mass
in $P\in\mathbb{R}^{n}$ and define 
\[
\mathcal{A}_{\mathcal{Q}}(\mathbb{R}^{n}):=\{\stackrel[i=1]{\mathcal{Q}}{\sum}[[P_{i}]]|P_{i}\in\mathbb{R}^{n},1\le i\le\mathcal{Q}\}.
\]
We endow $\mathcal{A}_{\mathcal{Q}}(\text{\ensuremath{\mathbb{R}}}^{n})$
with a metric $\mathcal{G}$, defined as follows: For each $T_{1},T_{2}\in\mathcal{A}_{\mathcal{Q}}(\text{\ensuremath{\mathbb{R}}}^{n})$
with $T_{1}=\stackrel[\mathit{i=\mathrm{1}}]{\mathcal{Q}}{\sum}[[P_{i}]],T_{2}=\stackrel[\mathit{i=\mathrm{1}}]{\mathcal{Q}}{\sum}[[S_{i}]]$
we define 
\[
\mathcal{G\mathrm{(\mathit{T_{\mathrm{1}},T_{\mathrm{2}}})\coloneqq\underset{\sigma\in\mathscr{P_{\mathcal{Q}}}}{min}\sqrt{\stackrel[\mathit{i=\mathrm{1}}]{\mathcal{Q}}{\sum}|\mathit{P_{i}-S_{\sigma\mathrm{(}i\mathrm{)}}}|^{2}}}},
\]
 where $\mathscr{P_{\mathcal{Q}}}$ is the permutation group on $\{1,...,\mathcal{Q}\}$.
The space $(\mathcal{A}_{\mathcal{Q}}(\text{\ensuremath{\mathbb{R}}}^{n}),\mathcal{G})$
is a complete metric space. A $\mathcal{Q}$- valued function is a
function $f:\Omega\subset\mathbb{R}^{m}\rightarrow\mathcal{A}_{\mathcal{Q}}(\text{\ensuremath{\mathbb{R}}}^{n})$.
Of course, this formalism was designed to capture the notion of a
function attaining multiple values at each point. A regularity theory
can be developed for $\mathcal{Q}$-valued functions. In particular,
the notion of a Sobolev space $W^{1,p}(\Omega,\mathcal{A}_{\mathcal{Q}}(\text{\ensuremath{\mathbb{R}}}^{n}))$
and the notion of an approximate differential denoted by $Df$. 

Suppose $\Omega\subset\mathbb{R}^{m}$ is a bounded domain with smooth
boundary. By analogy with the Dirichlet principle we say that a function
$f:\Omega\rightarrow\mathcal{A}_{\mathcal{Q}}(\text{\ensuremath{\mathbb{R}}}^{n})$
is Dir-minimizing if $f\in W^{1,2}(\Omega,\mathcal{A}_{\mathcal{Q}}(\text{\ensuremath{\mathbb{R}}}^{n}))$
and 

\[
\underset{\Omega}{\int}|Df|^{2}\leq\underset{\Omega}{\int}|Dg|^{2},
\]
 for all $g\in W^{1,2}(\Omega,\mathcal{A}_{\mathcal{Q}}(\text{\ensuremath{\mathbb{R}}}^{n})$
whose trace on $\partial\Omega$ agrees with that of $f$. We shall
always assume $m\geq2$. 

\subsection{Frequency Function}

We recall the frequency function and its relatives in the context
of $\mathcal{Q}-$valued functions. We have the following Holder regularity
type theorem for Dir-minimizing functions:

\begin{thm}

\textsl{\label{Theorem 2.17} }\textup{({[}5{]}, Theorem} \textup{6.2)}\textsl{
}There are constants $\alpha=\alpha(m,\mathcal{Q})\in(0,1)$ and $C=C(m,n,\mathcal{Q\mathrm{,}\delta\mathit{\mathrm{)}}}$
with the following property. If $f:B_{1}(0)\subset\mathbb{R}^{m}\rightarrow\mathcal{A}_{\mathcal{Q}}\mathrm{(\mathbb{R}^{\mathit{n}})}$
is Dir-minimizing then 
\[
\underset{x\neq y\in\overline{B_{\delta}(0)}}{\sup}\frac{\mathcal{G\mathrm{(\mathit{f\mathrm{(\mathit{x\mathrm{),\mathit{f\mathrm{(\mathit{y}))}}}}}}}}}{|x-y|^{\alpha}}\leq CDir(f)^{\frac{1}{2}}
\]
 for all $0<\delta<1$. 

\textup{In light of Theorem (\ref{Theorem 2.17}) $|f|^{2}$ is continuous
on $B_{1}(0)$. Fix $0<\delta<1$. Then, according to Theorem (\ref{Theorem 2.17})
$f$ is continuous on $\overline{B_{\delta}(0)}$. Since $\mathcal{G}$
is a metric, by the triangle inequality: 
\[
||T|-|S||=|\mathcal{G}(T,\mathcal{Q}[[0]])-\mathcal{G}(S,\mathcal{Q}[[0]])|\leq\mathcal{G}(T,S).
\]
 So $|f|$ is the composition of $f$ with a Lipschitz function, which
implies that $|f|$ is continuous on $\overline{B_{\delta}(0)}$,
and so the same is true for $|f|^{2}$. This is true for all $0<\delta<1$
from which we deduce continuity on $B_{1}(0).$ Thus, both $H_{f}$
and $\overline{H}_{f}$ are well defined for all $r\in(0,1)$. If
furthermore $H_{f}(r)>0$ for all $r$ then 
\[
I_{f}(r)\coloneqq\frac{rD_{f}(r)}{H_{f}(r)}
\]
 is well defined and is called the frequency function of $f$. When
there is no ambiguity, we shall omit the subscript $f$. }

\end{thm}
\begin{doublespace}

\section{\label{Section 4} Proof of Main Results}
\end{doublespace}

In this section we give a proof of Proposition (\ref{Proposition 1.2})
and Theorem (\ref{Theorem 1.2}). The following identities will play
a crucial role in the study of convexity of the frequency function
and its relatives
\begin{prop}
\textsl{\label{Proposition 3.1} }\textup{({[}5{]}, Proposition 5.2)}\textsl{
}Let $f:B_{R}(x)\subset\mathbb{R^{\mathit{m}}\rightarrow\mathcal{A_{Q}\mathrm{(\mathbb{R}^{\mathit{n}})}}}$
be Dir-minimizing.

Then for a.e. $0<r<R$ we have: (i) $(m-2)$$\underset{B_{r}(x)}{\int}|Df|^{2}dx=r\underset{\partial B_{r}(x)}{\int}|Df|^{2}d\sigma-2r\underset{\partial B_{r}(x)}{\int}\stackrel[i=1]{\mathcal{Q}}{\sum}|\partial_{\nu}f_{i}|^{2}d\sigma$. 

(ii) $\underset{B_{r}(x)}{\int}|Df|^{2}dx=\underset{\partial B_{r}(x)}{\int}\stackrel[i=1]{\mathcal{Q}}{\sum}(\partial_{\nu}f_{i})\cdot f_{i}d\sigma$. 
\end{prop}
Our starting point is the following theorem
\begin{thm}
\textup{({[}5{]}, Theorem 7.5)}\textsl{ \label{Theorem 3.2} }Let
$f:B_{1}(0)\subset\mathbb{R^{\mathit{m}}}\rightarrow\mathcal{A}_{\mathcal{Q}}(\text{\ensuremath{\mathbb{R}}}^{n})$
be Dir minimizing. Then:
\end{thm}
\textit{(a) $H\in C^{1}(0,1)$ and the following identity holds for
all $r\in(0,1):$ }

\textit{
\begin{equation}
H'(r)=\frac{m-1}{r}H(r)+2D(r).\label{eq:3.2}
\end{equation}
}

\textit{(b) If $H(r)>0$ then $I(r)$ is absolutely continuous and
non decreasing. In particular, $I'(r)\geq0$ for a.e. $r$. }

Statement (b) is known as Almgren's monotonicity formula. Since $D$
is absolutely continuous, it is a.e. differentiable. Therefore, in
view of equation (\ref{eq:3.2}), we see that $H'$ is a.e. differentiable.
Otherwise put, the second derivative of $H$ exists a.e. The regularity
properties of $H$ clearly apply for $\overline{H}$ as well. With
the aid of Almgren's monotonicity formula we are able to extend Agmon's
convexity result {[}1{]} in the context of $\mathcal{Q}-$valued
functions. This method of proof differs from Agmon's original approach
for real valued harmonic functions, which involves ODEs on Banach
spaces. 

\textit{Proof of Proposition 1.2. }In light of the above discussion
it is clear that $a$ is $C^{1}(-\infty,0)$ and that $a''$ exists
almost everywhere. For (i), note that 
\[
C_{m}\overline{H}'(r)=(\frac{1}{r^{m-1}}H(r))'=\frac{1}{r^{m-1}}H'(r)-(m-1)r^{-m}H(r)
\]

\[
=\frac{1}{r^{m-1}}(\frac{m-1}{r}H(r)+2D(r))-(m-1)\frac{1}{r^{m}}H(r)=\frac{2D(r)}{r^{m-1}}\geq0,
\]

Where the second equality is due to equation (\ref{eq:3.2}). So

\begin{equation}
\overline{H}'(r)=\frac{2D(r)}{C_{m}r^{m-1}}\geq0.\label{eq:7}
\end{equation}

Therefore, for all $t\in(-\infty,0)$ 

\begin{equation}
a'(t)=\log(\overline{H}(e^{t}))'=\frac{\overline{H}'(e^{t})e^{t}}{\overline{H}(e^{t})}\geq0.\label{eq:6-1}
\end{equation}
For (ii), start by noting that 
\[
a(t)=\log(\overline{H}(e^{t}))=\log(H(e^{t}))-\log(C_{m}e^{t(m-1)})
\]

\[
=\log(H(e^{t}))-\log(C_{m})-(m-1)t.
\]
Therefore, to prove $a''(t)\geq0$ it suffices to prove $(\log(H(e^{t})))''\geq0$.
By Theorem (\ref{Theorem 3.2}), we have that $I'(r)\geq0$ for a.e.
$0<r<1.$ To spare some space, all equalities and inequalities from
now on should be interpreted up to a null set. By virtue of equation
(\ref{eq:3.2})
\[
0\leq I'(r)=(\frac{rD(r)}{H(r)})'=(\frac{D(r)}{r^{m-2}\overline{H}(r)})'
\]

\[
=\frac{1}{2}(\frac{r(H'(r)-(\frac{m-1}{r})H(r))}{H(r)})'
\]

\[
=\frac{1}{2}(\frac{rH'(r)-(m-1)H(r)}{H(r)})'
\]

\[
=\frac{1}{2}(\frac{rH'(r)}{H(r)})'=\frac{1}{2}(\frac{(H'(r)+rH''(r))H(r)-r(H'(r))^{2}}{H^{2}(r)}).
\]

Thus, we get 
\begin{equation}
(H'(r)+rH''(r))H(r)-r(H'(r))^{2}\geq0.\label{eq:5}
\end{equation}
 On the other hand by a straightforward calculation:

\begin{equation}
e^{-t}(\log(H(e^{t})))''=(H'(e^{t})+e^{t}H''(e^{t}))H(e^{t})-e^{t}(H'(e^{t}))^{2}.\label{eq:6}
\end{equation}
Combining inequality (\ref{eq:5}) with equation (\ref{eq:6}) we
arrive at $e^{-t}(\log(H(e^{t})))''\geq0$ which is the same as $(\log(H(e^{t})))''\geq0$.
We are left to explain why $a$ is convex. It is classical that a
continuously differentiable function is convex iff its derivative
is non decreasing, and so our task reduces to showing that $a'$ is
non decreasing. By equations (\ref{eq:6-1}) and (\ref{eq:7}) we
get 
\[
a'(t)=\frac{\overline{H}'(e^{t})e^{t}}{\overline{H}(e^{t})}=\frac{2D(e^{t})}{C_{m}e^{t(m-2)}\overline{H}(e^{t})}.
\]
Since $D(e^{t})$ is a composition of an absolutely continuous function
with a non decreasing smooth function, it is absolutely continuous.
In addition, $\frac{1}{C_{m}e^{t(m-2)}\overline{H}(e^{t})}$ is differentiable.
So $a'(t)$ is absolutely continuous function on any closed sub interval
of $(-\infty,0)$, as a product of such function. Therefore, the fundamental
theorem of calculus is applicable: if $t_{1},t_{2}\in(-\infty,0),t_{1}<t_{2}$
then 
\[
a'(t_{2})-a'(t_{1})=\stackrel[t_{1}]{t_{2}}{\int}a''(t)dt\geq0.
\]
 
\begin{flushright}
$\square$
\par\end{flushright}
\begin{rem}
We draw the reader's attention to a somewhat delicate point, which
will be also relevant in what will come next . The implication ``non
negative derivative a.e.$\Rightarrow$non decreasing'' is not true
in general. In Proposition (\ref{Proposition 1.2}) we employed the
fact that the first derivative of $a$ is absolutely continuous in
order to deduce that it is convex. The absolute continuity of the
derivatives is a consequence of equation (\ref{eq:3.2}). Hence, we
implicitly relied here on the Dir-minimization property. In the case
of 1 valued harmonic functions, this technicality is not created because
all functions involved are smooth. To the best of our knowledge, improved
regularity for the frequency function and its relatives in the multi
valued settings is still an open problem. 
\end{rem}
The inequality ``$\overline{H}''(r)\geq0$ a.e.'' is not true in
general, as witnessed by the counterexample in section (\ref{Section 5}).
Nevertheless, we are still able to obtain a convexity result by reducing
the power of the normalization of $H$. More precisely we observe
the weaker
\begin{prop}
\label{Proposition 4.4} Let $f:B_{1}(0)\subset\mathbb{R^{\mathit{m}}}\rightarrow\mathcal{A}_{\mathcal{Q}}(\text{\ensuremath{\mathbb{R}}}^{n})$
be Dir minimizing. Then 

(i) $(r\overline{H}(r))'\geq0$ for all $r\in(0,1)$. (ii) $(r\overline{H}(r))''\geq0$
for a.e. $r\in(0,1)$. Furthermore $r\mapsto r\overline{H}(r)$ is
convex. 
\end{prop}
\textit{Proof.} As for (i) we can in fact derive a stronger result,
namely $\overline{H}'(r)\geq0$ for all $r\in(0,1)$. We compute:
\[
\overline{H}'(r)=\frac{1}{C_{m}}(-(m-1)r^{-m}H(r)+r^{-(m-1)}H'(r))
\]
\begin{equation}
=\frac{r^{-(m-1)}}{C_{m}}(H'(r)-\frac{m-1}{r}H(r))=\frac{2D(r)r^{-(m-1)}}{C_{m}}\geq0,\label{eq:3.4}
\end{equation}

where the last equality is by equation (\ref{eq:3.2}). For (ii),
start by noting that equation (\ref{eq:3.4}) implies the following
equality for a.e. $r$ 

\textit{
\begin{equation}
\overline{H}''(r)=\frac{2}{C_{m}r^{m-1}}(D'(r)-\frac{(m-1)D(r)}{r}).\label{eq:3.4'}
\end{equation}
}

Gathering our calculations one readily checks that $(r\overline{H}(r))''\geq0\iff rD'(r)-(m-3)D(r)\geq0$.
Indeed 
\[
rD'(r)-(m-3)D(r)\geq rD'(r)-(m-2)D(r)
\]

\[
=r\underset{\partial B_{r}(0)}{\int}|Df|^{2}-(m-2)\underset{B_{r}(0)}{\int}|Df|^{2}=2r\underset{\partial B_{r}(0)}{\int}\stackrel[i=1]{\mathcal{Q}}{\sum}|\partial_{\nu}f_{i}|^{2}\geq0,
\]
 where the last equality is thanks to (i), (\ref{Proposition 3.1}).
The convexity of $r\mapsto r\overline{H}(r)$ follows by a similar
argument to the one demonstrated in Proposition (\ref{Proposition 1.2}).
\begin{flushright}
$\square$
\par\end{flushright}

It is clear that Proposition (\ref{Proposition 4.4}) implies in particular
that the same conclusion holds true for $H$ (and this can also be
derived directly from iterating equation (\ref{eq:3.2})).

We present now a proof of Theorem (\ref{Theorem 1.2}), including
a higher dimensional analog. The main ingredients of the proof are
the variational formulas provided by Proposition (\ref{Proposition 3.1})
and the following estimates 
\begin{prop}
\textsl{\label{Proposition 3.6} }\textup{({[}5{]}, Proposition 6.3)
}Let $f:B_{1}(0)\subset\mathbb{R^{\mathit{m}}}\rightarrow\mathcal{A}_{\mathcal{Q}}\mathrm{(\mathbb{R}^{\mathit{n}})}$
be Dir-minimizing and suppose that $g_{r}:=f|_{\partial B_{r}(0)}\in W^{1,2}(\partial B_{r}(0),\mathcal{A}_{\mathcal{Q}}\mathrm{(\mathbb{R}^{\mathit{n}}))}$
for a.e $0<r<1$. Then for a.e $r$: (i) If $m=2$, $\mathrm{Dir}(f,B_{r}(0))\leq\mathcal{Q}r\mathrm{Dir}(g_{r},\partial B_{r}(0))$.
(ii) If $m>2$, $\mathrm{Dir}(f,B_{r}(0))\leq c(m)r\mathrm{Dir}(g_{r},\partial B_{r}(0))$,
where $c(m)<\frac{1}{m-2}$.
\end{prop}
Before going into the proof of Theorem (\ref{Theorem 1.2}), let us
heuristically explain why it is reasonable the expect the validity
of such a result through the following example.
\begin{example}
\label{Example 4.6} We recall that a function $f:B_{1}(0)\rightarrow\mathcal{A}_{\mathcal{Q}}\mathrm{(\mathbb{R}^{\mathit{n}})}$
is $\alpha-$homogeneous ($\alpha>0)$ if $\forall y\in B_{1},y\neq0:f(y)=|y|^{\alpha}f(\frac{y}{|y|})$.
Denote by $\eta:\mathcal{A}_{\mathcal{Q}}\mathrm{(\mathbb{R}^{\mathit{n}})}\rightarrow\mathbb{R}^{n}$
the center of mass map defined by 
\[
\eta(\stackrel[i=1]{\mathcal{Q}}{\sum}[[P_{i}]])\coloneqq\frac{\stackrel[i=1]{\mathcal{Q}}{\sum}P_{i}}{\mathcal{Q}}.
\]
We can derive an explicit formula for the $L^{2}$-growth function
of a continuous $\alpha-$homogeneous map:

\[
H(r)=\underset{\partial B_{r}}{\int}|f|^{2}d\sigma=r^{m-1}\underset{\partial B_{1}}{\int}|f|^{2}(ry)d\sigma
\]

\[
=r^{m-1}\underset{\partial B_{1}}{\int}\stackrel[i=1]{\mathcal{Q}}{\sum}|f_{i}|^{2}(ry)d\sigma=r^{2\alpha+m-1}\underset{\partial B_{1}}{\int}\stackrel[i=1]{\mathcal{Q}}{\sum}|f_{i}|^{2}(y)d\sigma=\kappa r^{2\alpha+m-1},
\]

for some constant $\kappa\geq0$. Hence $\overline{H}$ takes the
form $\overline{H}(r)=\kappa r^{2\alpha}$. Assume now that $m=2$,
$f:B_{1}(0)\rightarrow\mathcal{A}_{\mathcal{Q}}\mathrm{(\mathbb{R}^{\mathit{n}})}$
is a Dir-minimizing, nontrivial, $\alpha-$homogeneous map with $\eta\circ f=0$
(the simplest example of such a map is the 2-valued function $f:B_{1}(0)\rightarrow\mathcal{A}_{2}(\mathbb{R}^{2})$
defined by $f(z)=\underset{w^{2}=z}{\sum}[[w]]$. See Theorem (\ref{Theorem 5.1})).
It is proved in {[}5{]}, Proposition 8.2 that in this case necessarily
$\alpha=\frac{p}{q}\in\mathbb{Q}$ for some $q\leq\mathcal{Q}$. So
$\overline{H}(r)=\kappa r^{\frac{2p}{q}}$, which implies $\overline{H}(r^{\frac{q}{2}})=\kappa r^{p}$,
and the latter function is obviously absolutely monotonic. We are
thus lead to speculate that more generally, composing $\overline{H}$
with some suitable $\frac{1}{2}[\mathcal{Q}]$-power produces a function
which is more well behaved. Theorem (\ref{Theorem 1.2}) partially
confirms this speculation. 
\end{example}
\textit{Proof of Theorem \ref{Theorem 1.2}.} That $\overline{h}'_{\frac{\mathcal{Q}}{2}}(r)\geq0$
follows immediately from equation (\ref{eq:7}). Henceforth all equalities
and inequalities should be interpreted up to a null set. A direct
calculation gives:

\[
\overline{h}''_{N}(r)=N(N-1)r^{N-2}\overline{H}'(r^{N})+\overline{H}''(r^{N})N^{2}r^{2N-2}.
\]

Writing $\xi=r^{N}$, we see that 
\[
\overline{h}_{N}''(\xi)\geq0\Leftrightarrow(\frac{N-1}{N})\overline{H}'(\xi)+\overline{H}''(\xi)\xi\geq0.
\]
 Taking $N=\frac{\mathcal{Q}}{2}$ we get 
\[
\overline{h}_{\frac{\mathcal{Q}}{2}}''(\xi)\geq0\Leftrightarrow(\mathcal{Q}-2)\overline{H}'(\xi)+\mathcal{Q}\overline{H}''(\xi)\xi\geq0.
\]

Owing to equation (\ref{eq:7}) we can express $\overline{H}',\overline{H}''$
explicitly

\[
\overline{H}'(\xi)=\frac{2}{C\xi}\underset{B_{\xi}(0)}{\int}|Df|^{2}
\]

and

\[
\overline{H}''(\xi)=\frac{2}{C\xi}(\underset{\partial B_{\xi}(0)}{\int}|Df|^{2}-\frac{1}{\xi}\underset{B_{\xi}(0)}{\int}|Df|^{2}).
\]

We now have

\[
C\xi((\mathcal{Q}-2)\overline{H}'(\xi)+\mathcal{Q}\overline{H}''(\xi)\xi)
\]

\[
=2(\mathcal{Q}-2)\underset{B_{\xi}(0)}{\int}|Df|^{2}+2\mathcal{Q}\xi(\underset{\partial B_{\xi}(0)}{\int}|Df|^{2}-\frac{1}{\xi}\underset{B_{\xi}(0)}{\int}|Df|^{2})
\]

\begin{equation}
=2\mathcal{Q}\xi\underset{\partial B_{\xi}(0)}{\int}|Df|^{2}-4\underset{B_{\xi}(0)}{\int}|Df|^{2}.\label{eq:3.7}
\end{equation}

Proposition (\ref{Proposition 3.6}), (i) combined with Proposition
(\ref{Proposition 3.1}) yield the following estimate

\[
\underset{B_{\xi}(0)}{\int}|Df|^{2}dx\leq\xi\mathcal{Q}\mathrm{Dir}(g_{\xi},\partial B_{\xi}(0)).
\]

Therefore 
\[
2\underset{B_{\xi}(0)}{\int}|Df|^{2}\leq2\xi\mathcal{Q}\mathrm{Dir}(g_{\xi},\partial B_{\xi}(0))
\]

\[
=2\xi\mathcal{Q}(\underset{\partial B_{\xi}(0)}{\int}|Df|^{2}-\stackrel[i=1]{\mathcal{Q}}{\sum}|\partial_{\nu}f_{i}|^{2}d\sigma)=\mathcal{Q}\xi\underset{\partial B_{\xi}(0)}{\int}|Df|^{2},
\]

where the last equality is due to (\ref{Proposition 3.1}), (i). 

Combining the last inequality with equation (\ref{eq:3.7}) gives
$(\mathcal{Q}-2)\overline{H}'(\xi)+\mathcal{Q}\overline{H}''(\xi)\xi\geq0$,
as wanted. Finally, note that 
\[
\overline{h}'_{\frac{\mathcal{Q}}{2}}(r)=\frac{2D(r^{\frac{\mathcal{Q}}{2}})r^{-\frac{\mathcal{Q}}{2}(m-1)}}{C_{m}}r^{\frac{\mathcal{Q}}{2}-1}=\frac{2D(r^{\frac{\mathcal{Q}}{2}})r^{\frac{\mathcal{Q}}{2}(2-m)-1}}{C_{m}}.
\]
 Since $D(r^{\frac{\mathcal{Q}}{2}})$ is a composition of absolutely
continuous non decreasing functions, it is absolutely continuous.
In view of the previous equation we see that $\overline{h}'_{\frac{\mathcal{Q}}{2}}$
is absolutely continuous. Thus, proceeding as in Proposition (\ref{Proposition 1.2}),
it follows that $\overline{h}{}_{\frac{\mathcal{Q}}{2}}$ is convex.
\begin{flushright}
$\square$
\par\end{flushright}

We finish this section by stating an higher dimensional analog of
(\ref{Proposition 3.6}) 
\begin{thm}
\textsl{\label{Theorem 3.14} }Let $m>2$ and $f:B_{1}(0)\subset\mathbb{R}^{m}\rightarrow\mathcal{A}_{\mathcal{Q}}(\mathbb{R}^{\mathit{n}})$
be a Dir minimizing function. Suppose $f|_{\partial B_{r}}\in W^{1,2}(\partial B_{r},\mathcal{A}_{\mathcal{Q}}(\text{\ensuremath{\mathbb{R}}}^{n}))$
for a.e. $0<r<1$. Let $c(m)=\frac{1}{m-2}-\epsilon_{m}$ , $0<\epsilon_{m}<\frac{1}{m-2}$
be the constant obtained via proposition \ref{Proposition 3.6}. Let
$\alpha_{m}=\frac{1-\epsilon_{m}(m-2)}{\epsilon_{m}(m-2)^{2}}$. Then

(i) $\overline{h}'_{\frac{\alpha_{m}}{2},f}(r)\geq0$ for all $r\in(0,1)$.

(ii) $\overline{h}''_{\frac{\alpha_{m}}{2},f}(r)\geq0$ for a.e. $r\in(0,1)$.
\end{thm}
\begin{proof}
The proof is identical to that of Theorem (\ref{Theorem 1.2}), using
estimate (ii) in Proposition (\ref{Proposition 3.6}) instead of (i). 
\end{proof}
We can now conclude that\textsl{ }nontrivial Dir minimizing, $\alpha-$homogeneous
function have exponents $\alpha$ far away from $0$. 
\begin{cor}
\label{Corollary 4.8} There are constants $\beta_{m}>0$ with the
following property. Let $f:B_{1}(0)\subset\mathbb{R}^{m}\rightarrow\mathcal{A}_{\mathcal{Q}}(\mathbb{R}^{\mathit{n}})$
be a nontrivial Dir minimizing, $\alpha-$homogeneous function. Suppose
$f|_{\partial B_{r}}\in W^{1,2}(\partial B_{r},\mathcal{A}_{\mathcal{Q}}(\text{\ensuremath{\mathbb{R}}}^{n}))$
for a.e. $0<r<1$. Then $\alpha\geq\beta_{m}$. 
\end{cor}
\begin{proof}
Put $\beta_{2}=\frac{1}{\mathcal{Q}}$ and $\beta_{m}=\frac{1}{\alpha_{m}},m>2$.
According to the computation performed in Example \ref{Example 4.6},
$\overline{H}(r)=\kappa r^{2\alpha}$ for some $\kappa>0$. According
to Theorem (\ref{Theorem 1.2}) and Theorem (\ref{Theorem 3.14})
we see that $0\leq\kappa\frac{\alpha}{\beta_{m}}(\frac{\alpha}{\beta_{m}}-1)r^{\frac{\alpha}{\beta_{m}}-2}$
for a.e. $r$, which in particular gives $\alpha\geq\beta_{m}$. 
\end{proof}
Note that in the specific case that $m=2$ and $\eta\circ f=0$, Corollary
(\ref{Corollary 4.8}) recovers a weaker form of Proposition 8.2,
{[}5{]} which was already mentioned in Example (\ref{Example 4.6}). 
\begin{doublespace}

\section{\label{Section 5} Counterexamples}
\end{doublespace}

The following theorem allows us to produce nontrivial examples of 

Dir-minimizing functions:
\begin{thm}
\textsl{\label{Theorem 5.1} }\textup{({[}2{]}, Theorem 2.20) }Let
$a\neq0,b\in\mathbb{R}$. Define $u:B_{1}(0)\subset\mathbb{R}^{2}\mathbb{\rightarrow}\mathcal{A_{\mathcal{Q}}}(\mathbb{R}^{2})$
by $u(z)=\underset{w^{\mathcal{Q}}=az+b}{\sum}[[w]]$. Then $u$ is
Dir-minimizing.

\textup{The original proof of Theorem (\ref{Theorem 5.1}) in {[}2{]}
is highly nontrivial and relies heavily on the theory of mass minimizing
currents. Spadaro ({[}8{]}, theorem 0.1) found an alternative simpler
proof to this result.  We start by demonstrating that Theorem (\ref{Theorem 1.1 })
cannot be naively extended to the $\mathcal{Q}-$valued setting.}
\end{thm}
\begin{example}
Define $f:B_{1}(0)\subset\mathbb{R}^{2}\mathbb{\rightarrow\mathcal{A}_{\mathrm{3}}\mathrm{(\mathbb{R^{\mathrm{2}}\mathrm{)}}}}$
by $f(z)=\underset{w^{3}=z}{\sum}[[w]]$. That $f$ is Dir minimizing
follows from Theorem (\ref{Theorem 5.1}). We compute
\end{example}
\[
\underset{B_{\rho}(0)}{\int}|f|^{2}dx=\stackrel[0]{2\pi}{\int}\stackrel[0]{\rho}{\int}3r^{\frac{5}{3}}drd\theta=3\stackrel[0]{2\pi}{\int}\stackrel[0]{\rho}{\int}r^{\frac{5}{3}}drd\theta=\frac{9\pi\rho^{\frac{8}{3}}}{4}=C\rho^{\frac{8}{3}},
\]
 where $C>0$. Therefore, $H(\rho)=C\rho^{\frac{5}{3}}$ for some
constant $C>0$, and so $\overline{H}(\rho)=\frac{C\rho^{\frac{5}{3}}}{2\pi\rho}=C\rho^{\frac{2}{3}}$,
hence $\overline{H}''(\rho)<0$.

Our next aim is to show that the boundary regularity condition appearing
in Proposition (\ref{Proposition 1.4}) is indeed verified for a certain
Dir-minimizing functions on the planar unit disk. This will be the
content of Lemma (\ref{Lemma 5.4}), which is of interest by its own
right. Any $z\in\mathbb{C}-\{0\}$ admits a representation of the
form $z=Re^{i\omega}$ for some $R>0,\omega\in[0,2\pi)$. We shall
use the convention $\sqrt{z}=\sqrt{R}e^{\frac{i\omega}{2}}$. 
\begin{lem}
\label{Lemma 5.2 } Fix $r\in(0,\frac{1}{2})\cup(\frac{1}{2},1)$
and define $h:(0,2\pi)\rightarrow\mathbb{C}$ by 
\[
h(\theta)\coloneqq\sqrt{2re^{i\theta}-1}.
\]
 Then $h\in W^{1,2}((0,2\pi),\mathbb{C})$. 

Proof. \textup{Trivially $h\in L^{2}(0,2\pi)$. Furthermore, 
\[
h'(\theta)=\frac{2rie^{i\theta}}{2\sqrt{2re^{i\theta}-1}}=\frac{rie^{i\theta}}{\sqrt{2re^{i\theta}-1}},
\]
 hence 
\[
|h'(\theta)|^{2}=\frac{r}{|2re^{i\theta}-1|}.
\]
 Since $|2re^{i\theta}-1|>0$ for all $\theta\in[0,2\pi]$, there
is some $C_{r}>0$ such that $|2re^{i\theta}-1|>C_{r}$ for all $\theta\in(0,2\pi]$,
hence the asserted. }

\end{lem}
We recall that $f\in W^{1,2}(\Omega,\mathcal{A}_{\mathcal{Q}}(\mathbb{R}^{n}))$
if there are $\{\varphi_{j}\}_{j=1}^{m}\subset L^{2}(\Omega,\mathbb{R}_{\geq0})$
such that 1. $x\mapsto\mathcal{G}(f(x),T)\in W^{1,2}(\Omega)$ for
all $T\in\mathcal{A}_{\mathcal{Q}}(\mathbb{R}^{n})$ and 2. $|\partial_{j}\mathcal{G}(f(x),T)|\leq\varphi_{j}$
a.e. for all $T\in\mathcal{A}_{\mathcal{Q}}(\mathbb{R}^{n})$ and
$1\leq j\leq m$. 
\begin{lem}
\label{Lemma 5.4} Fix $r\in(0,\frac{1}{2})\cup(\frac{1}{2},1)$ and
define $f:(0,2\pi)\rightarrow\mathcal{A}_{2}(\mathbb{R}^{2})$ by
\[
f(\theta)\coloneqq\underset{z^{2}=2re^{i\theta}-1}{\sum}[[z]].
\]
 Then $f\in W^{1,2}((0,2\pi),\mathcal{A}_{2}(\mathbb{R}^{2}))${\large{}.}{\large\par}
\end{lem}
\textit{Proof. }Let $T=\stackrel[i=1]{2}{\sum}[[T_{i}]]\in\mathcal{A}_{2}(\mathbb{R}^{2})$.
Set $h(\theta)=\sqrt{2re^{i\theta}-1}$ and for $P=(P_{1},P_{2})\in\mathbb{R}^{2}\times\mathbb{R}^{2}$
denote by $d_{P}:\mathbb{R}^{2}\rightarrow\mathbb{R}$ the function
\[
d_{P}(x)\coloneqq\sqrt{|x-P_{1}|^{2}+|x-P_{2}|^{2}}.
\]
 It is not difficult to see that for any fixed $P$, $d_{P}$ is Lipschitz.
Put 
\[
\alpha_{T}(\theta)\coloneqq(d_{(-T_{1},T_{2})}\circ h)(\theta),\beta_{T}(\theta)=(d_{(T_{1},-T_{2})}\circ h)(\theta).
\]
Note that 

\begin{equation}
\mathcal{G}(f(\theta),T)=\frac{\alpha_{T}(\theta)+\beta_{T}(\theta)-|\alpha_{T}(\theta)-\beta_{T}(\theta)|}{2}.\label{eq:}
\end{equation}

By Lemma (\ref{Lemma 5.2 }), $\alpha_{T},\beta_{T}$ are a composition
of a Lipschitz function on a $W^{1,2}(0,2\pi)$ function. Therefore
$\alpha_{T},\beta_{T}$ are also $W^{1,2}(0,2\pi)$. In view of Equation
(\ref{eq:}) it is now apparent that $\theta\mapsto\mathcal{G}(f(\theta),T)\in W^{1,2}(0,2\pi)$.
In addition, there is at most one $\theta_{0}\in(0,2\pi)$ for which
$\alpha_{T}(\theta_{0})=0$. An elementary calculation shows that
the following estimate is obeyed for $\theta\in(0,2\pi)-\{\theta_{0}\}$:
\[
|(\partial_{\theta}\alpha_{T})(\theta)|\leq2|(\partial_{\theta}h_{1})(\theta)|+2|(\partial_{\theta}h_{2})(\theta)|
\]
and $\partial_{\theta}h_{i}\in L^{2}(0,2\pi)$, $i=1,2$ by Lemma
(\ref{Lemma 5.2 }). The same is true for $\beta_{T}$. Combining
these estimates with equation (\ref{eq:}) we see that 

\[
|\partial_{\theta}\mathcal{G}(f(\theta),T)|\leq8(|(\partial_{\theta}h_{1})(\theta)|+|(\partial_{\theta}h_{2})(\theta)|)\in L^{2}(0,2\pi),
\]

for all but finitely many $\theta\in(0,2\pi)$. So strictly by definition,
\[
f\in W^{1,2}((0,2\pi),\mathcal{A}_{2}(\mathbb{R}^{2})).
\]
 
\begin{flushright}
$\square$
\par\end{flushright}

\textit{Proof of Proposition}\textsl{ \ref{Proposition 1.4}.} That
$f$ is Dir minimizing follows from Theorem (\ref{Theorem 5.1}).
Furthermore, by Lemma (\ref{Lemma 5.4}) $f|_{\partial B_{r}}\in W^{1,2}(\partial B_{r},\mathcal{A}_{\mathrm{2}}\mathrm{(\mathbb{R^{\mathrm{2}}\mathrm{)}}})$
for all $r\neq\frac{1}{2}$.

We compute:
\[
\underset{B_{\rho}(0)}{\int}|f|^{2}dx=2\stackrel[0]{\rho}{\int}\stackrel[0]{2\pi}{\int}r\sqrt{(2r\cos\theta-1)^{2}+4r^{2}\sin^{2}(\theta)}d\theta dr.
\]
 Therefore 
\[
H(\rho)=2\stackrel[0]{2\pi}{\int}\rho\sqrt{(2\rho\cos\theta-1)^{2}+4\rho^{2}\sin^{2}(\theta)}d\theta,
\]

which implies 
\[
\overline{H}(\rho)=\frac{1}{\pi}\stackrel[0]{2\pi}{\int}\sqrt{(2\rho\cos\theta-1)^{2}+4\rho^{2}\sin^{2}(\theta)}d\theta
\]

\[
=\frac{1}{\pi}\stackrel[0]{2\pi}{\int}\sqrt{1+4\rho^{2}-4\rho\cos\theta}d\theta\coloneqq\frac{1}{\pi}A(\rho).
\]

We compute $A'''(\rho)$ for all $\frac{1}{2}<\rho<1$:

\[
A'(\rho)=\stackrel[0]{2\pi}{\int}\frac{8\rho-4\cos\theta}{2\sqrt{1+4\rho^{2}-4\rho\cos\theta}}d\theta=\stackrel[0]{2\pi}{\int}\frac{4\rho-2\cos\theta}{\sqrt{1+4\rho^{2}-4\rho\cos\theta}}d\theta
\]
.

\[
A''(\rho)=\stackrel[0]{2\pi}{\int}\frac{4\sqrt{1+4\rho^{2}-4\rho\cos\theta}-\frac{2(4\rho-2\cos\theta)^{2}}{2\sqrt{1+4\rho^{2}-4\rho\cos\theta}}}{1+4\rho^{2}-4\rho\cos\theta}d\theta
\]

\[
=\stackrel[0]{2\pi}{\int}\frac{4\sqrt{1+4\rho^{2}-4\rho\cos\theta}-\frac{(4\rho-2\cos\theta)^{2}}{\sqrt{1+4\rho^{2}-4\rho\cos\theta}}}{1+4\rho^{2}-4\rho\cos\theta}d\theta
\]

\[
=\stackrel[0]{2\pi}{\int}\frac{4(1+4\rho^{2}-4\rho\cos\theta)-(4\rho-2\cos\theta)^{2}}{(1+4\rho^{2}-4\rho\cos\theta)^{\frac{3}{2}}}d\theta
\]

\[
=\stackrel[0]{2\pi}{\int}\frac{4\sin^{2}\theta}{(1+4\rho^{2}-4\rho\cos\theta)^{\frac{3}{2}}}d\theta.
\]

\[
A'''(\rho)=\stackrel[0]{2\pi}{\int}-6\sin^{2}\theta(1+4\rho^{2}-4\rho\cos\theta)^{-\frac{5}{2}}(8\rho-4\cos\theta)d\theta
\]

\begin{equation}
=-24\stackrel[0]{2\pi}{\int}\sin^{2}\theta(1+4\rho^{2}-4\rho\cos\theta)^{-\frac{5}{2}}(2\rho-\cos\theta)d\theta.\label{eq:2}
\end{equation}

We note that as long as $\frac{1}{2}<\rho<1$, the RHS of equation
(\ref{eq:2}) is strictly negative, hence the claim. 
\begin{flushright}
$\square$
\par\end{flushright}
\begin{doublespace}

\section{Appendix}
\end{doublespace}

This section was explained to me by Dan Mangoubi. We present here
a proof of Theorem (\ref{Theorem 1.1 }).
\begin{prop}
\textsl{\label{Proposition 5.1 } }Let $u$ be harmonic on $B_{1}(0)$.
Let $\phi:[0,1]\rightarrow\mathbb{R}$ be smooth and non decreasing.
Define $d(r)=\underset{B_{1}(0)}{\int}u^{2}(rx)\phi'(||x||^{2})dx$.
Then $d$ is absolutely monotonic. 
\end{prop}
\begin{proof}
We denote by $u_{i}$ the derivative of $u$ with respect to the $i-$th
variable. Define $\psi(\xi)=\phi(\frac{||\xi||^{2}}{r^{2}})$. 
\[
d'(r)=\underset{\text{\ensuremath{B_{1}(0)}}}{{\normalcolor \int}}(2\stackrel[\text{\ensuremath{i}=1}]{m}{\sum}u(rx)u_{i}(rx)x_{i})\phi'(||x||\text{\texttwosuperior})dx
\]

\[
=\underset{B_{r}(0)}{\int}(2\stackrel[\text{\ensuremath{i}=1}]{m}{\sum}u(\xi)u_{i}(\xi)\frac{\xi_{i}}{r^{m+1}})\phi'(\frac{||\xi||^{2}}{r^{2}})d\xi
\]

\[
=\frac{1}{r^{m-1}}\underset{B_{r}(0)}{\int}(2\stackrel[\text{\ensuremath{i}=1}]{m}{\sum}u(\xi)u_{i}(\xi)\frac{\xi_{i}}{r^{2}})\phi'(\frac{||\xi||^{2}}{r^{2}})d\xi=\frac{1}{2r^{m-1}}\underset{B_{r}(0)}{\int}\nabla u^{2}\cdot\nabla\psi d\xi
\]

\[
\overset{\ast}{=}\frac{1}{2r^{m-1}}(\underset{\partial B_{r}(0)}{\int}\psi\nabla_{\nu}u^{2}d\sigma-\underset{B_{r}(0)}{\int}\psi\Delta u^{2}d\xi)
\]

\[
=\frac{1}{2r^{m-1}}(\underset{\partial B_{r}(0)}{\int}\phi(1)\nabla_{\nu}u^{2}d\sigma-\underset{B_{r}(0)}{\int}\psi\Delta u^{2}d\xi)
\]

\[
\overset{\ast\ast}{=}\frac{1}{2r^{m-1}}(\underset{B_{r}(0)}{\int}\phi(1)\Delta u^{2}d\xi-\underset{B_{r}(0)}{\int}\psi\Delta u^{2}d\xi)=\frac{1}{2r^{m-1}}\underset{B_{r}(0)}{\int}(\phi(1)-\psi)\Delta u^{2}d\xi.
\]

The equality $\ast$ is by Green's identity and the equality $\ast\ast$
is by the divergence theorem. Taking advantage of the identity $\Delta u^{2}=|\nabla u|^{2}$
for $u$ harmonic we obtain
\[
d'(r)=\frac{1}{2r^{m-1}}\underset{B_{r}(0)}{\int}(\phi(1)-\psi)|\nabla u|^{2}d\xi=\frac{1}{2}r\underset{B_{1}(0)}{\int}(\phi(1)-\phi(||x||^{2}))|\nabla u|^{2}(rx)dx.
\]
Let $\Phi$ be some anti derivative of $\phi$ and define $\varphi:[0,1]\rightarrow$\ensuremath{\mathbb{R}}
by $\varphi(\rho)\coloneqq\phi(1)\rho-\Phi(\rho)$. Evidently $\varphi$
is non decreasing and according to the computation we preformed 

\[
d'(r)=\frac{1}{2}r\underset{B_{1}(0)}{\int}|\text{\ensuremath{\nabla}}u|^{2}(rx)\varphi'(||x||\text{\texttwosuperior})dx.
\]

Since each $u_{i}$ is harmonic we can iterate the the same argument
in order to obtain $d^{(k)}(r)\ge0$ for all $k$. 
\end{proof}
As a corollary we obtain a proof of theorem (\ref{Theorem 1.1 })
\begin{cor}
Let $u$ be harmonic on $B_{1}(0)$. Then $\overline{H}$ is absolutely
monotonic.
\end{cor}
\textit{Proof. }
\[
\overline{H}(r)=\frac{1}{C_{m}r^{m-1}}\underset{\partial B_{r}(0)}{\int}u^{2}(x)d\sigma
\]

\[
=\frac{1}{C_{m}r^{m-1}}\frac{d}{dr}(\underset{B_{r}(0)}{\int}u^{2}(x)dx)=\frac{1}{C_{m}r^{m-1}}\frac{d}{dr}(\underset{B_{r}(0)}{\int}u^{2}(x)dx)
\]

\[
=\frac{1}{C_{m}r^{m-1}}\frac{d}{dr}(r^{m}\underset{B_{1}(0)}{\int}u^{2}(rx)dx)
\]

\[
=\frac{1}{C_{m}r^{m-1}}(mr^{m-1}\underset{B_{1}(0)}{\int}u^{2}(rx)dx+r^{m}\frac{d}{dr}(\underset{B_{1}(0)}{\int}u^{2}(rx)dx))
\]

\[
=\frac{1}{C_{m}}(m\underset{B_{1}(0)}{\int}u^{2}(rx)dx+r\frac{d}{dr}(\underset{B_{1}(0)}{\int}u^{2}(rx)dx)).
\]

Taking $\phi(t)=t$ in Proposition (\ref{Proposition 5.1 }), we see
that last expression is a sum of absolutely monotonic functions, and
hence absolutely monotonic. 
\begin{flushright}
$\square$
\par\end{flushright}

\medskip{}

\textbf{Acknowledgment}. This work is part of the author's MA thesis.
I wish to express my most deep gratitude towards my supervisor Dan
Mangoubi for suggesting and guiding me in this problem, for many fruitful
and stimulating discussions and for carefully reading previous versions
of this manuscript and providing insightful comments. In particular,
for explaining the proof of Theorem (\ref{Theorem 1.1 }). I would
also like to thank Camillo De Lellis for his interest in this work.
In particular, for suggesting the counterexample that appears in Proposition
(\ref{Proposition 1.4}) and for explaining the proof of Theorem (\ref{Theorem 5.1}).
I thank an anonymous referee for various comments which improved the
quality of this work.

\end{document}